\documentclass[12pt]{article}
\usepackage{amsmath}
\usepackage{amssymb}
\usepackage{amscd}
\usepackage{amsthm}

\theoremstyle{plain}
\newtheorem{Thm}{Theorem}

\newtheorem{Prop}[Thm]{Proposition}
\newtheorem{Cor}[Thm]{Corollary}
\newtheorem{Lem}[Thm]{Lemma}

\theoremstyle{definition}

\theoremstyle{Remark}

\numberwithin{equation}{section}

\title{On the abundance theorem in the case $\nu=0$}
\author{Yujiro Kawamata}

\begin{document}

\maketitle

\begin{abstract}
We present a short proof of the abundance theorem in the case of 
numerical Kodaira dimension $0$ proved by Nakayama and its log generalizaton.
\end{abstract}

\section{Introduction}

Nakayama \cite{Nakayama} proved the abundance conjecture for a non-minimal 
algebraic variety whose numerical Kodaira dimension is equal to $0$:

\begin{Thm}\label{main}
Let $X$ be a smooth projective variety.
Assume that the function $\dim H^0(X, mK_X + A)$ is bounded when 
$m \to \infty$ for arbitrarily fixed ample divisor $A$.
Then there exists a positive integer $m$ such that 
$H^0(X,mK_X) \ne 0$. 
\end{Thm}

Nakayama's result is more general in the sense that the theorem holds
for KLT pairs.
Siu \cite{Siu} proved the same result by analytic method. 

The purpose of this note to present a simplified version of the 
proof.
The main point is to use the numerical version of the 
Zariski decomposition as in \cite{cone} and Simpson's finiteness result
\cite{Simpson}.

We shall also prove a logarithmic generalization of Theorem~\ref{main}
for normal crossing pairs (Theorem~\ref{log}).
We note that the coefficients of the boundary in this case are equal to $1$ 
and the pair is not KLT.

Abundance theorem in the case $\nu=0$ for minimal algebraic varieties 
is already proved in 
\cite{addition} as an application of the additivity theorem of the 
Kodaira dimension for algebraic fiber spaces.

We work over $\mathbf{C}$.
We denote by $\equiv$ and $\sim$ the numerical 
and linear equivalences respectively.

When this paper was posted on the web, 
the author learned from a message by Frederic Campana that the argument using
\cite{Simpson} already appeared in \cite{CPT}.

\section{Numerical Zariski decomposition}

Let $X$ be a smooth projective variety.
Two $\mathbf{R}$ divisors $D$ and $D'$ on $X$ are said to be 
{\em numerically equivalent}, and denoted by $D \equiv D'$,
if $(D \cdot C) = (D' \cdot C)$ for 
all irreducible curves $C$.
The set of all numerical classes of $\mathbf{R}$-divisors form a 
finite dimensional real vector space $N^1(X)$.
The {\em pseudo-effective cone} 
$\text{Pseff}(X)$ is the smallest closed convex cone 
in $N^1(X)$ which contains all the numerical classes of effective divisors.
An $\mathbf{R}$-divisors is said to be {\em pseudo-effective} if its 
numerical class is contained in $\text{Pseff}(X)$.
The {\em movable cone} 
$\text{Mov}(X)$ is the smallest closed convex cone 
in $N^1(X)$ which contains all the numerical classes of effective divisors
whose complete linear systems do not have fixed components.

\begin{Lem}
Let $D$ be an $\mathbf{R}$-divisor and $A$ an ample divisor
on a smooth projective variety $X$.
The following are equivalent:

(1) $D$ is pseudo-effective.

(2) For an aribtrary positive number $\epsilon$, 
there exists an effective $\mathbf{R}$-divisor $D'$ such that 
$D+\epsilon A \equiv D'$.

(3) For an aribtrary positive number $\epsilon$, 
there exists a positive integer $m$ such that 
$H^0(X,\llcorner m(D+\epsilon A) \lrcorner) \ne 0$.
\end{Lem}

\begin{proof}
The equivalence of (1) and (2) is clear.
Obviously (3) implies (2).
If (2) holds, then we can write 
\[
D+\frac 13\epsilon A - D' = \sum_j d_jD_j
\]
for an effective $\mathbf{R}$-divosor $D'$, 
real numbers $d_j$ which are linearly independent over $\mathbf{Q}$
and $\mathbf{Q}$-divisors $D_j$ such that $D_j \equiv 0$.
Thus we can write
\[
D+\frac 13\epsilon A \sim_{\mathbf{Q}} D_1 + D_2 + L
\]
where $\sim_{\mathbf{Q}}$ denotes the $\mathbf{Q}$-linear equivalence,
for an effective $\mathbf{Q}$-divisor $D_1$, 
a small $\mathbf{R}$-divisor $D_2$
and $L \in \text{Pic}^0(X) \otimes \mathbf{Q}$ such that
$\frac 13\epsilon A + D_2$ and $\frac 13\epsilon A + L$ are 
$\mathbf{Q}$-linearly equivalent to effective $\mathbf{R}$-divisors.
Then there exists a positive integer $m$ such that 
$m(D+\epsilon A)$ is linearly equivalent to an effective $\mathbf{R}$-divisor, 
hence (3).
\end{proof}

Let $D$ be a pseudo-effective $\mathbf{R}$-divisor, and 
$A$ an ample divisor.
We define the {\em numerically fixed part} and the 
{\em numerical base locus} of $D$ by
\[
\begin{split}
&N(D) = \lim_{\epsilon \downarrow 0} 
(\inf \{D' \,\vert\, D+\epsilon A \equiv D' \ge 0\}) \\
&\text{NBs}(D) = \bigcup_{\epsilon > 0} 
(\bigcap \{\text{Supp}(D') \,\vert\, D+\epsilon A \equiv D' \ge 0\}).
\end{split}
\]
They are independent of $A$.
By setting $P(D) = D - N(D)$, we obtain a formula $D = P(D)+N(D)$ 
called the {\em numerical Zariski decomposition} of $D$.

\begin{Lem}
(1) The irreducible components of $N(D)$ are numerically 
independent, i.e., linearly independent in $N^1(X)$.

(2) If $D$ is a $\mathbf{Q}$-divisor and $D \equiv N(D)$, then 
$N(D)$ is also a $\mathbf{Q}$-divisor.

(3) $N(D)=0$ if and only if the numerical class of $D$ is 
contained in $\text{Mov}(X)$.
In this case, $\text{NBs}(D)$ is a countable union of subvarieties of 
codimension at least $2$.
\end{Lem}

\begin{proof}
(1) If there is a numerical linear relation, 
then $N(D)$ is numerically equivalent to a different effective 
$\mathbf{R}$-divisor, a contradiction.

(2) The intersection numbers of $N(D)$ with curves are rational numbers, 
hence so are the coefficients of $N(D)$.

(3) We can take the limit $\epsilon \to 0$ for only those $\epsilon$ which are 
rational numbers.
\end{proof}

Let $D$ be a pseudo-effective $\mathbf{R}$-divisor.
Then there are two cases:

\begin{enumerate}
\item $\nu(X,D) = 0$:
The function $\dim H^0(X, \llcorner mD \lrcorner + A)$ 
is bounded when $m \to \infty$ for any ample divisor $A$.

\item $\nu(X,D) > 0$:
There exists an ample divisor $A$ such that the function
$\dim H^0(X, \llcorner mD \lrcorner + A)$ 
is unbounded when $m \to \infty$.
\end{enumerate}

The following proposition is \cite{Nakayama}~Theorem~V.1.11.
We include a proof for the convenience of the reader.

\begin{Prop}
Let $X$ be a smooth projective variety, and 
$D$ a pseudo-effective $\mathbf{R}$-divisor. 
Assume that $D \not\equiv 0$ and $N(D) = 0$.
Then there exist an ample divisor $A$, a positive number $b$ and a positive 
integer $m_0$ such that 
\[
\dim H^0(X, \llcorner mD \lrcorner + A) > bm
\]
for $m \ge m_0$.
\end{Prop}

\begin{proof}
Since $\text{NBs}(D)$ is a countable union of closed 
subvarieties of codimension at least $2$, 
a general curve section $C$ does not meet $\text{NBs}(D)$.
Since $D \not\equiv 0$, we have $(D \cdot C) > 0$.

We fix an ample divisor $A$ and denote $L_m = \llcorner mD \lrcorner + A$.
It is sufficient to prove that the natural homomorphism 
$H^0(X, L_m) \to H^0(C, L_m \vert_C)$ is surjective for $m$ large.

Let $\mu: Y \to X$ be the blowing up along $C$ and $E$ the exceptional divisor.
We take an effective $\mathbf{R}$-divisor 
$B_m \equiv mD+ \epsilon A$ such that 
$C \not\subset B_m$ and that $(Y, \mu^*B_m)$ is KLT near $E$.
We calculate
\[
\begin{split}
&\mu^*L_m - E - (K_Y+\mu^*B_m) \\
&= \mu^*(\llcorner mD \lrcorner + A - (K_X+B_m)) - (n-1)E \\
&\equiv \mu^*((1-\epsilon)A - \langle mD \rangle - K_X) - (n-1)E
\end{split}
\]
where $n=\dim X$ and $\langle mD \rangle = mD - \llcorner mD \lrcorner$.
It is ample for any $m > 0$ if $A$ is sufficiently large compared to the 
irreducible components of $D$, $K_X$ and $E$. 

Let $I$ be the multiplier ideal sheaf for the pair $(Y, \mu^*B_m)$.
We have $E \cap \text{Supp}(\mathcal{O}_Y/I) = \emptyset$.
By the Nadel vanishing theorem, we have $H^1(Y, I(\mu^*L_m - E)) = 0$. 
It follows that the homomorphism 
$H^0(Y, \mu^*L_m) \to H^0(E, \mu^*L_m \vert_E)$ is surjective, and 
our assertion is proved.
\end{proof}

\begin{Cor}
Let $X$ be a smooth projective variety, and 
$D$ a pseudo-effective $\mathbf{R}$-divisor. 
Assume that $\dim H^0(X, \llcorner mD \lrcorner + A)$ is bounded.
Then $D$ is numerically equivalent to an effective $\mathbf{R}$-divisor $N(D)$.
\end{Cor}

When $D = K_X$, we denote $\nu(X) = \nu(X,K_X)$.
The Kodaira dimension $\kappa(X)$ is a birational invariant.
Its numerical version $\nu(X)$ is also a birational invariant: 
if $X$ and $X'$ are birationally equivalent 
smooth projective varieties, then $\nu(X) = 0$ if and only if $\nu(X')=0$.
For more precise definition of the numerical Kodaira dimension $\nu(X)$,
we refer the reader to \cite{Nakayama}.

\section{Proof of the theorem}

\begin{proof}[Proof of Theorem~\ref{main}]
By assumption, $K_X$ is numerically equivalent to an effective 
$\mathbf{Q}$-divisor.
We take the smallest possible positive integer $m$ such that 
$m(K_X+L)$ for some $L \in \text{Pic}^{\tau}(X)$
is linearly equivalent to an effective divisor $N$.
We shall prove that $L$ is a torsion in $\text{Pic}^{\tau}(X)$.

By blowing up $X$ further, 
we may assume that the support of $N$ is a normal crossing divisor.
We take a holomorphic section $h$ of $\mathcal{O}_X(m(K_X+L))$ 
such that $\text{div}(h) = N$.
By taking the $m$-th root of $h$, we construct a finite and surjective 
morphism $\pi: Y' \to X$ from a normal variety with only 
rational singularities.
Let $\mu: Y \to Y'$ be a desingularization.
We have
\[
\mu_*K_Y \sim \pi^*(K_X+(m-1)(K_X+L) - N')
\]
for some $\mathbf{Q}$-divisor $N'$ such that $0 \le N' \le N$.
Then 
\[
\mu_*K_Y+\pi^*L \sim \pi^*(m(K_X+L) - N') \sim \pi^*(N-N').
\]
Since $Y'$ has only rational singularities, it follows that 
$H^0(Y,K_Y + \mu^*\pi^*L) \ne 0$.

By \cite{Simpson}, it follows that 
there exists a torsion element $L' \in \text{Pic}^{\tau}(Y)$
such that $H^0(Y,K_Y +L') \ne 0$.
If $L' \not\sim \mu^*\pi^*L$, then $\pi^*(N-N')$ is numerically equivalent 
to a different effective divisor.
Then $N$ must be numerically equivalent to a different effective 
$\mathbf{Q}$-divisor, a contradiction.
Therefore $L' \sim \mu^*\pi^*L$, and $L$ is a torsion.
\end{proof}

We prove a logarithmic version:

\begin{Thm}\label{log}
Let $X$ be a smooth projective variety and $D = \sum_i D_i$ a simple 
normal crossing divisor.
Assume that $\dim H^0(X, m(K_X+D) + A)$ is bounded when $m \to \infty$
for any fixed ample divisor $A$.
Then there exists a positive integer $m$ such that 
$H^0(X,m(K_X+D)) \ne 0$. 
\end{Thm}

\begin{proof}
The proof is parallel to the non-log case.
We have $m(K_X+D+L) \sim N$ as before.
We make the union of $N$ and $D$ to be normal crossing by blowing up $X$, 
and take a ramified covering $\pi: Y' \to X$ branching along $N$.
By resolution, we obtain a smooth 
projective variety $Y$ with a simple normal crossing divisor $E$ which is 
the union of the preimage of $D$ and
the exceptional divisors of the resolution.
We note that common irreducible components of $N$ and $D$ do not cause
any trouble for the formula
\[
\mu_*(K_Y+E) \sim \pi^*(K_X+D+(m-1)(K_X+D+L) - N')
\]
though we have to modify $N'$.
We have $\mu_*(K_Y+E) + \pi^*L \sim \pi^*(N-N')$ as before.
We have to prove that $\pi^*L$ is torsion. 

In the moduli space of local systems $V$ of rank $1$ on $X$, we consider 
the closed subvarieties where the dimensions of the cohomology groups 
\[
H^p_B(Y \setminus E,V) = H^p(Y \setminus E, \mu^*\pi^*V)
\]
jump.
Let $i: Y \setminus E \to Y$ be an open immersion, and denote
by $E^{[p]}$ the disjoint union of all the $p$-fold intersections 
of the irreducible components of $E$ as in \cite{Deligne}.
We have $E^{[0]} = Y$ by convention.
Then the canonical filtration on the complex $Ri_*\mathbf{Q}_{Y \setminus E}$
induces a spectral sequence among Betti cohomologies
\[
E_1^{p,q} = H^{2p+q}_B(E^{[-p]}, V) \Rightarrow H_B^{p+q}(Y \setminus E, V)
\]
where we denote $H^q_B(E^{[p]}, V) = H^q(E^{[p]}, \mu^*\pi^*V)$.
We denote furthermore
\[
\begin{split}
&H^q_{DR}(E^{[p]}, V) = 
\mathbf{H}^q(E^{[p]}, \mu^*\pi^*V \otimes \Omega_{E^{[p]}}^{\bullet}) \\
&H^q_{DR}(Y \setminus E, V) = 
\mathbf{H}^q(Y, \mu^*\pi^*V \otimes \Omega_Y^{\bullet}(\log E)) \\
&H^q_{\text{Dol}}(E^{[p]}, V) = 
\mathbf{H}^q(E^{[p]}, (\mu^*\pi^*V \otimes \Omega_{E^{[p]}}^{\bullet}, \phi)) 
\\
&H^q_{\text{Dol}}(Y \setminus E, V) = 
\mathbf{H}^q(Y, (\mu^*\pi^*V \otimes \Omega_Y^{\bullet}(\log E), \phi))
\end{split}
\]
where $\phi \in H^0(X, \Omega^1_X)$ is the Higgs field corresponding to the
flat connection on $V$.
The filtration with respect to the orders of log poles on the complex 
$\Omega_Y^{\bullet}(\log E)$ induces spectral sequences 
\[
\begin{split}
&E_1^{p,q} = H^{2p+q}_{DR}(E^{[-p]}, V)
\Rightarrow H^{p+q}_{DR}(Y \setminus E, V) \\
&E_1^{p,q} = H^{2p+q}_{\text{Dol}}(E^{[-p]}, V)
\Rightarrow H^{p+q}_{\text{Dol}}(Y \setminus E, V)
\end{split}
\]
By \cite{Deligne}, these spectral sequences are compatible with the 
isomorphisms 
\[
H^q_B(E^{[p]}, V) \cong H^q_{DR}(E^{[p]}, V) 
\cong H^q_{\text{Dol}}(E^{[p]}, V)
\]
so that we have isomorphisms 
\[
H^q_B(Y \setminus E, V) \cong H^q_{DR}(Y \setminus E, V) 
\cong H^q_{\text{Dol}}(Y \setminus E, V).
\]
Thus our jumping loci are canonically defined in the sense of 
Simpson \cite{Simpson}.
We apply Simpson's result, and deduce 
that there is a torsion element $L' \in \text{Pic}^{\tau}(X)$ such that 
$H^0(Y, L' \otimes \mathcal{O}_Y(K_Y+E)) \ne 0$.
The rest of the proof is the same as in Theorem~\ref{main}.
\end{proof}

\section{Addendum}

After the first version of this paper is written,
the author received a paper \cite{COP}.
Then the author realized that a more precise calculation on the 
construction of this paper yields a more general result for 
LC pairs as follows.
The proof is added for the sake of completeness though
it is basically the same.

\begin{Thm}
Let $(X,B)$ be a projective pair with log canonical (LC) 
singularities.
Assume that $H^0(X,m(K_X+B+L)) \ne 0$ for a positive 
integer $m$ and $L \in \text{Pic}^{\tau}(X)$.
Then there exists a positive integer $m'$ 
such that $H^0(X,m'(K_X+B)) \ne 0$.  
\end{Thm}

\begin{proof}
We have $m(K_X+B+L) \sim N$ for an effective divisor $N$ as before.
By a resolution of singularities, we may assume that the support of 
$B+N$ is a normal crossing divisor.
Moreover we may assume that $B$ and $N$ have no common irreducible 
components by subtracting the overlap and multiplying $m$ if necessary.
We denote $B = D+B'$ for $D= \llcorner B \lrcorner$.
Let $\pi: Y' \to X$ be the ramified covering obtained by 
taking the $m$-th root of $N-mB$, and $\mu: Y \to Y'$
a log resolution as before.
Let $f = \pi\mu$.

Let $E = (f^*D)_{\text{red}}$.
Then we have $R\mu_*\mathcal{O}_Y(K_Y+E) \cong 
\mathcal{O}_{Y'}(K_{Y'}+\pi^*D)$ by the vanishing theorem.
We have the following more precise formula:
\[
Rf_*\mathcal{O}_Y(K_Y+E) \cong \bigoplus_{i=0}^{m-1} 
\mathcal{O}_X(K_X+D+\ulcorner i(K_X+B+L-N/m) \urcorner).
\]
Since $H^0(X, \bullet)$ of each term on the right hand side 
is upper semicontinuous, the jumping locus on the moduli space of  
flat line bundles on $X$ for each term is a union of 
torsion translations of triple tori by Simpson's result as before.

If we set $i=m-1$, then we have
\[
K_X+D+\ulcorner (m-1)(K_X+B+L-N/m) \urcorner
= m(K_X+B) + (m-1)L - \llcorner N/m \lrcorner.
\]
Therefore we conclude that there exists a torsion line bundle $L'$
such that $H^0(X, m(K_X+B)+L'- \llcorner N/m \lrcorner) \ne 0$.
\end{proof}

Department of Mathematical Sciences, University of Tokyo,

Komaba, Meguro, Tokyo, 153-8914, Japan

kawamata@ms.u-tokyo.ac.jp

\end{document}